\title{Symplectic log Calabi-Yau surface---deformation class}

\documentclass[11pt]{article}

\usepackage{amsfonts}
\usepackage{amsthm}
\usepackage{latexsym, amssymb, bbm}
\usepackage{xypic}
\usepackage{xcolor}
\usepackage{graphicx}
\usepackage{amsmath}
\usepackage{enumerate}
\usepackage[all]{xy}
\usepackage[margin=1.3in]{geometry}

\usepackage{amsmath,amsthm,amsfonts,amssymb,graphicx,psfrag}
\usepackage{color}
\usepackage{hyperref}

\newtheorem{thm}{Theorem}[section]

\newtheorem{prop}[thm]{Proposition}
\newtheorem{lemma}[thm]{Lemma}

\newtheorem{defn}[thm]{Definition}

\newtheorem{eg}[thm]{Example}

\begin{document}

\author{Tian-Jun Li and Cheuk Yu Mak
}
\AtEndDocument{\bigskip{\footnotesize
  \textsc{School of Mathematics, University of Minnesota, Minneapolis, MN 55455} \par
  \textit{E-mail address}, Tian-Jun Li: \texttt{tjli@math.umn.edu} \par
  \addvspace{\medskipamount}
  \textsc{School of Mathematics, University of Minnesota, Minneapolis, MN 55455} \par
  \textit{E-mail address}, Cheuk Yu Mak: \texttt{makxx041@math.umn.edu} \par
}}
\date{October 20, 2015}
\maketitle
\begin{center}
{\it Dedicated to Professor Shing-Tung Yau on the occasion of his 65th birthday}
\end{center}

\vspace{2em}

\begin{abstract}
We study the symplectic analogue of log Calabi-Yau surfaces
and show that the symplectic deformation classes of these surfaces are completely determined by the homological information.
\end{abstract}

\tableofcontents

\section{Introduction}
In \cite{Au07} and \cite{GrHaKe11}, Auroux and Gross-Hacking-Keel proposed a way to interpret mirror symmetry for Looijenga pair $(X,D)$, where $X$ is a smooth projective surface over $\mathbb{C}$ and $D$ is an effective reduced anti-canonical divisor on $X$ with maximal boundary.
Under mirror symmetry, certain symplectic invariants of $X-D$ are conjectured to be related to holomorphic invariants of its mirror.
In this regard, Pascaleff showed in \cite{Pa13} that the symplectic cohomology of $X-D$ is, as a vector space, isomorphic to the global sections of the structure sheaf of its mirror.
A step towards a deeper understanding of mirror symmetry for Looijenga pairs would be to classify them.
The moduli spaces of such pairs were studied by Looijenga in \cite{Lo81} and Gross-Hacking-Keel in \cite{GrHaKe12}.
Friedman gave an excellent survery in \cite{Fr15}.
Since one direction of mirror symmetry concerns about the symplectic invariants of $X-D$ instead of the holomorphic invariants,
we would like to establish, in this paper, a classification for `symplectic log Calabi-Yau surfaces' (including `symplectic Looijenga pairs' as a special case).
From symplectic point of view, we have the following definition of log Calabi-Yau surfaces.

For a connected closed symplectic 4 dimensional manifold $(X, \omega)$, which we assume throughout the whole paper, a {\bf symplectic divisor} $D$ is a
connected configuration of finitely many closed embedded symplectic surfaces (called irreducible components) $D=C_1 \cup \dots \cup C_k$.
$D$ is further required to have the following two properties:
No three different $C_i$ intersect at a point and any intersection between two irreducible components is transversal and positive.
The orientation of each $C_i$ is chosen to be positive with respect to $\omega$.


\begin{defn} A {\bf symplectic log Calabi-Yau surface} $(X,D,\omega)$ is a closed symplectic real dimension
four manifold $(X,\omega)$ together with a symplectic divisor $D$ representing the homology class of the Poincare dual of $c_1(X, \omega)$.

A  symplectic Looijenga pair $(X,D,\omega)$ is a symplectic log Calabi-Yau surface such that each irreducible component of $D$ is a sphere.
\end{defn}

Let $(X,D,\omega)$ be a symplectic log Calabi-Yau surface. By Theorem A of \cite{Liu96} or \cite{OhOn96}
and the adjunction formula,  it is easy to show  (Lemma \ref{lemma: genus 1})
that  $X$ is uniruled with base genus $0$ or $1$, and
 $D$ is a torus or a cycle of spheres.
And if $(X,D,\omega)$ is a symplectic Looijenga pair then $X$ is rational.




Similar to studying the moduli space under complex deformation in the complex category, we would like to classify symplectic log Calabi-Yau surfaces up to symplectic deformation equivalence.

\begin{defn} \label{} A  {\bf symplectic homotopy} (resp. {\bf symplectic isotopy}) of $(X, D, \omega)$  is a smooth one-parameter family of symplectic divisors  $(X, D_t,\omega_t)$ with
$(X, D_0,\omega_0)=(X, D, \omega)$
(resp. such that in addition $\omega_t=\omega$ for all $t$).
 $(X',D',\omega')$ is  said to be  {\bf symplectic deformation equivalent} to $(X,D,\omega)$ if
 it is symplectomorphic to $(X, D_1,\omega_1)$ for some symplectic homotopy $(X, D_t,\omega_t)$ of $(X, D, \omega)$.
 The symplectic deformation equivalence is called {\bf strict} if  the symplectic homotopy is a
 symplectic isotopy.
\end{defn}

\begin{defn}  Two symplectic log Calabi-Yau surfaces $(X^i,D^i,\omega^i)$ for $i=1,2$ are said to be {\bf homological equivalent} if there is a
diffeomorphsim $\Phi: X^1 \to X^2$ such that $\Phi_*[C^1_j]=[C^2_j]$ for all $j=1,\dots,k$.
The homological equivalence is called {\bf strict} if  $\Phi^*[\omega^2]=[\omega^1]$.
We call $\Phi$ a (strict) homological equivalence.
\end{defn}

Here is the main result of this paper.

\begin{thm}\label{thm: symplectic deformation class=homology classes}
Let $(X^i,D^i,\omega^i)$ be symplectic log Calabi-Yau surfaces for $i=1,2$.
Then $(X^1,D^1,\omega^1)$ is (resp. strictly) symplectic deformation equivalent to $(X^2,D^2,\omega^2)$ if and only if they are (resp. strictly) homological equivalent.

Moreover, the symplectomorphism in the (resp. strict) symplectic deformation equivalence has same homological effect as the (resp. strict) homological equivalence.
\end{thm}

We remark that when $D$ is a smooth divisor, the relative Kodaira dimension $\kappa(X, D, \omega)$
 was
introduced in \cite{LiZh11} and it was noted there that this notion  could be
extended to nodal divisors. With this extension understood, symplectic Calabi-Yau surfaces
have relative Kodaira dimension $\kappa=0$  (cf. Theorem 3.28 in \cite{LiZh11}).
Moreover,
Theorem \ref{thm: symplectic deformation class=homology classes} is also valid when   $\kappa(X, D, \omega)=-\infty$.  This will be treated in the sequel.
Coupled with the techniques developed in \cite{LiMa14}, \cite{LiMaYa14}, some applications to Stein fillings will also be treated in the sequel.

The paper is organized as follows.
In   Section \ref{section: Classification of Symplectic Log Calabi-Yau surfaces}
we introduce marked divisors and establish the invariance of their deformation class
under blow-up/down in Proposition \ref{prop: marked blowdown}. This reduces
Theorem \ref{thm: symplectic deformation class=homology classes} to the minimal cases.
In Section \ref{section: Reduction to minimal models}, we classify the deformation classes of minimal models and finish the proof of Theorem \ref{thm: symplectic deformation class=homology classes}.


The authors benefit from discussions with Mark Gross, Paul Hacking and Sean Keel.
Both authors are supported by NSF-grants DMS 1065927 and 1207037.

\section{Symplectic deformation equivalence of marked divisors}\label{section: Classification of Symplectic Log Calabi-Yau surfaces}

We study the symplectic deformation equivalence property in a general setting,
which was  initiated  by Ohta and Ono in  \cite{OhOn05}. Here we provide details using the notion of marked
divisor,
which encodes the blow-down information.
We will  show that the deformation class of marked symplectic divisors is stable under various operations.

\subsection{Homotopy and blow-up/down  of symplectic divisors}

\subsubsection{Homotopy}
Parallel to the  two types of homotopy of a symplectic divisor $(X,D,\omega)$ mentioned in the introduction,

$\bullet$ Symplectic isotopy $(X,D_t,\omega)$, and

$\bullet$ Symplectic homotopy $(X,D_t,\omega_t)$.\\
We also consider the more restrictive homotopies keeping $D$ fixed:

$\bullet$ $D-$symplectic isotopy $(X,D,\omega_t)$ with constant $[\omega_t]$, and

$\bullet$ $D-$symplectic homotopy $(X,D,\omega_t)$

To compare these notions we introduce the following terminology.

 \begin{defn} Two {\it symplectic homotopies} are said to be symplectomorphic  if they are
related by a one parameter family of symplectomorphisms.
\end{defn}

\begin{lemma}\label{lem: def versus marked def}
A symplectic homotopy (resp. isotopy) of a symplectic divisor is symplectomorphic
to a $D-$symplectic homotopy (resp. isotopy) and vice versa.

\end{lemma}

\begin{proof}
A $D-$symplectic homotopy is a symplectic homotopy by definition, and  by Moser lemma
a $D-$symplectic isotopy is symplectomorphic to a symplectic isotopy.

On the other hand,  a symplectic homotopy $(X,D_t,\omega_t)$ gives rise to  a smooth isotopy
$\Phi:D\times [0, 1]\to X$.
Since the intersections of $D$ are transversal and no three of the components intersect at a common point,
we can apply the smooth isotopy extension theorem to extend $\Phi$ to a smooth  ambient
 isotopy
$\Phi=\{\Phi_t\}: X \times [0,1] \to X$.
Then we get a $D-$symplectic homotopy $(X,D,\Phi_t^*\omega_t)$ which is symplectomorphic to $(X,D_t,\omega_t)$ via the one parameter  family of symplectomorphisms $\{\Phi_t\}$.
Similarly, a symplectic isotopy is symplectomorphic to a $D-$symplectic isotopy.

\end{proof}

Lemma \ref{lem: def versus marked def} converts the effect of a symplectic isotopy (resp. homotopy) to a $D-$symplectic isotopy (resp. homotopy).
This  simple observation  will be repeatedly used.

\subsubsection{Toric and non-toric blow-up/down}

Throughout the  paper, we use the following terminology for symplectic blow-up/down of $D\subset (X, \omega)$.

A {\bf toric blow-up} (resp. {\bf non-toric blow-up}) of $D$ is the total (resp. proper)
transform of a symplectic blow-up centered at an intersection point (resp. at a smooth point) of $D$.

Here, for blow-up at a smooth point $p$ on the divisor $D$, it means that we first do a $C^0$ small perturbation of $D$ to $D'$ fixing $p$ and
then we do a symplectic blow-up of a ball centered at $p$ such that $D'$ coincide, in the local coordinates given by the ball,
with a complex subspace.
Similarly, for blow-up at an intersection point, a $C^0$ small perturbation is performed so that $D'$ is $\omega$-orthogonal at $p$
and $D'$ coincide, in the local coordinates given by the ball, with two complex subspaces.

To describe the corresponding blow-down operations, recall that a symplectic sphere with self-intersection $-1$ is called an exceptional sphere.
The homology class of an exceptional sphere is called an exceptional class.

A {\bf toric blow-down} refers to blowing down an exceptional sphere  contained in
$D$ that intersects exactly two other irreducible components and exactly once for each of them.
Moreover, we require that  the intersections are positive and transversal.
Such an exceptional sphere is called a toric exceptional sphere.

A {\bf non-toric blow-down} refers to blowing down an exceptional sphere not contained in $D$ that intersects exactly one irreducible component of $D$
and exactly once with the
intersection being positive and transversal.
Such an  exceptional sphere is called a  non-toric exceptional sphere.

More precisely, for blow-down of a toric or non-toric exceptional sphere $E$,
we first perturb our symplectic divisor $D$ to another symplectic divisor $D'$ (or perturbing $E$) such that the intersections of $D'$ and $E$
are $\omega$-orthogonal
(In the case that $E$ is an irreducible component of $D$, we require $E$ has $\omega$-orthogonal intersections with all other irreducible components).
Then, we will do the symplectic blow-down of $E$ and $D'$ will descend to a symplectic divisor.




\begin{defn}
An exceptional class $e$ is called {\bf non-toric} if $e$ has trivial intersection pairing with all but one of the
homology classes of the irreducible components of $D$ and the only non-trivial pairing is $1$.

An exceptional class $e$ is called {\bf toric} if $e$ is homologous to an irreducible component of $D$
such that $e$ pairs non-trivially with the classes of exactly two other irreducible components of $D$ and these two pairings are $1$.
\end{defn}

Clearly, the homology class of a toric (non-toric) exceptional sphere is a toric (non-toric) exceptional class.  Conversely, we have the following observations.

For a toric exceptional class $e$,  the component of $D$
with class $e$ is obviously a toric exceptional sphere in the class $e$.
For a non-toric exceptional class $e$, we also have an exceptional sphere in the class $e$, at least
when $D$ is $\omega-$orthogonal.
\begin{lemma}\label{lem: existence of nice J-sphere}(cf. Theorem 1.2.7 of \cite{McOp13})
 Let $D$ be an $\omega$-orthogonal symplectic divisor.
 There is a non-empty subspace $\mathcal{J}(D)$ of the space of $\omega$-tamed almost complex structure making $D$ pseudo-holomorphic such that
 for any non-toric exceptional class $e$, there is a residue subset $\mathcal{J}(D,e) \subset \mathcal{J}(D)$
 so that $e$ has an embedded $J$-holomorphic representative for all $J\in \mathcal{J}(D,e)$.
\end{lemma}

\begin{proof}
 It is immediate to prove that $e$ is $D$-good in the sense of Definition $1.2.4$ in \cite{McOp13} if $e$ is non-toric.
 Theorem $1.2.7$ of \cite{McOp13} then implies the result.
\end{proof}

\subsection{Deformation of marked divisors}\label{subsection: deformation problem for non-basic}


When we blow down an exceptional sphere, we  encode the process  by marking the descended symplectic divisor.

\begin{defn}
 A  marked symplectic divisor consists of a five-tuple
 $$\Theta=(X,D,\{p_j\}_{j=1}^l,\omega,\{I_j\}_{j=1}^l)$$
 such that


 $\bullet$ $D\subset (X, \omega)$ is a symplectic divisor,

 $\bullet$ $p_j$,  called centers of marking,
 are points on $D$ (intersection points of $D$ allowed),

 $\bullet$ $I_j: (B(\delta_j),\omega_{std}) \to (X,\omega)$, called   coordinates of  marking,
 are symplectic embeddings sending the origin  to $p_j$ and with $I_j^{-1}(D)=\{x_1=y_1=0\} \cap B(\delta_j)$
 (resp. $I_j^{-1}(D)=(\{x_1=y_1=0\} \cup \{x_2=y_2=0\})\cap B(\delta_j)$)
 if $p_j$ is a smooth (resp. an intersection) point of $D$.
  Moreover, we require that the image of $I_j$ are disjoint.
\end{defn}

If $p_j$ is an intersection point of $D$, then we define the symplectic embedding $I_j^{re}=I_j \circ re $,
where $re(x_1,y_1,x_2,y_2)=(-x_2,-y_2,x_1,y_1)$ interchanges the two subspaces $\{x_1=y_1=0\}$ and $\{x_2=y_2=0\}$.
If $p_j$ is a smooth point of $D$, then we define $I_j^{re}=I_j$.
For simplicity, we denote a marked symplectic divisor as $(X,D,p_j,\omega,I_j)$ or $\Theta$ and also call it a marked divisor if no confusion would arise.

\begin{defn}\label{defn: marked symplectic deformation}
Let $\Theta=(X,D,p_j,\omega,I_j)$ be a marked divisor.
A {\bf $D-$symplectic homotopy} (resp. {\bf $D-$symplectic isotopy}) of $\Theta$ is a 4-tuple $(X, D, p_j,\omega_t)$ such that
$\omega_t$ is a smooth family  of symplectic forms (resp. cohomologous symplectic forms) on $X$
with $\omega_0=\omega$ and
$D$ being $\omega_t$-symplectic  for all $t$.


If $\Theta^2=(X^2,D^2,p^2_j,\omega^2,I^2_j)$ is another marked symplectic divisor and
there is a symplectomorphism sending
the 4-tuple $(X^2,D^2,p^2_j,\omega^2)$ to a 4-tuple $(X,D,p_j,\omega_1)$ which is symplectic homotopic (isotopic) to $\Theta$, then we
say that $\Theta$ and $\Theta^2$
are {\bf $D-$symplectic deformation equivalent} (resp. {\bf strict $D-$symplectic deformation equivalent}).

\end{defn}

A symplectic divisor can be viewed as a marked divisor with empty markings.

\begin{lemma}\label{lem: as a marked divisor}
 Two symplectic divisors are
(strict) deformation equivalent if and only if they are (strict) D-deformation equivalent as marked symplectic divisor.
\end{lemma}

\begin{proof}
It follows directly from Lemma \ref{lem: def versus marked def}.
To obtain a (strict) D-symplectic deformation equivalence from a (strict) symplectic deformation equivalence, we just have to pre-compose
the symplectomorphism from $(X, D, \Phi_1^*\omega_1)$ to $(X, D_1, \omega_1)$.
The other direction is similar.
\end{proof}

For marked divisors, both $D-$symplectic deformation equivalence and its strict version  do not involve  the symplectic embeddings $I_j$.
We have the following seemingly stronger definition of deformation.

\begin{defn}
Let $\Theta=(X,D,p_j,\omega,I_j)$ be a marked divisor.
A {\bf strong $D-$symplectic homotopy} (resp. {\bf strong $D-$symplectic isotopy}) of $\Theta$ is a 5-tuple $(X, D, p_j,\omega_t,I_{j,t})$ such that

$\bullet$ the 4-tuple $(X, D, p_j,\omega_t)$ is a $D-$symplectic homotopy (resp. isotopy) of $\Theta$,

$\bullet$ $D$ is $\omega_t$-orthogonal, and

$\bullet$ $I_{j,t}: B(\epsilon_j) \to (X,\omega_t)$ are symplectic embedding sending the origin to $p_j$, $I_{j,0}=I_j|_{B(\epsilon_j)}$ and
$(I_{j,t})^{-1}(D) = \{x_1=y_1=0\} \cap B(\epsilon_j)$
(resp. $(I_{j,t})^{-1}(D) = (\{x_1=y_1=0\} \cup \{x_2=y_2=0\}) \cap B(\epsilon_j)$)
if $p_j$ is a smooth point (resp. $p_j$ is an intersection point), for some $\epsilon_j < \delta_j$.


If $\Theta^2=(X^2,D^2,p^2_j,\omega^2,I^2_j)$ is another marked sympelctic divisor and
there is a symplectomorphism sending
$(X^2,D^2,p^2_j,\omega^2,(I^2_j)^\#)$ to $(X,D,p_j,\omega_1,I_{j,1})$, where $(I^2_j)^\#$ is the unique choice between $I^2_j$
and $(I^2_j)^{re}$ such that the symplectomorphism is possible, then we
say that $\Theta$ and $\Theta^2$
are {\bf strong $D-$symplectic deformation equivalent} (resp. {\bf strong strict $D-$symplectic deformation equivalent}).
\end{defn}



\begin{lemma}\label{lem: D-deformation=strong D-deformation}
If $\Theta=(X,D,\{p_j\}_{j=1}^l,\omega,\{I_j\}_{j=1}^l)$ and
$\Theta^2=(X^2,D^2,\{p^2_j\}_{j=1}^l,\omega^2,\{I^2_j\}_{j=1}^l)$ are (strict) $D-$symplectic deformation equivalent,
then they are strong (strict) $D-$symplectic deformation equivalent.
\end{lemma}

\begin{proof}

We will only do the case when $l=1$. It can be done similarly for general $l$.
We denote $p_1$ as $p$, $I_1$ as $I$ and $I^2_1$ as $I^2$.

By assumption, there is a $D-$symplectic homotopy $(X,D, p,\omega_t)$ of $\Theta$
such that there is a symplectomorphism sending $(X,D, p,\omega_1)$ to $(X^2,D^2,p^2_1,\omega^2)$.
Therefore, without loss of generality, we can assume $(X,D, p,\omega_1)=(X^2,D^2,p^2_1,\omega^2)$.

The proof is easier when $p$ is a smooth point of $D$ so we only prove the case when $p$ is an intersection point of $D$.
Moreover, by possibly replacing $I^2$ with $(I^2)^{re}$, we can assume the irreducible component of $D$ corresponding to $\{x_1=y_1=0\}$ in chart $I$ is the same
as that of $I^2$.

The idea of the proof goes as follows.
First, we find a smooth family of symplectic embeddings of small ball $\Phi_t: (B(\delta),\omega_{std}) \to (X,\omega_t)$ sending the origin to $p$ such that
$\Phi_0=I|_{B(\delta)}$ and $\Phi_1=I^2|_{B(\delta)}$.
Then, we find another family of symplectic forms $\omega_t'$ such that the 4-tuple
$(X,D, p,\omega'_t)$ is still a $D-$symplectic homotopy of $\Theta$ with $\omega'_1=\omega_1$ and $D$ is $\omega'_t$-orthogonal for all $t$.
A corresponding symplectic embeddings  $I_{t}'$ for $(X,D, p,\omega'_t)$ will be constructed based on $\Phi_t$
such that the 5-tuple $(X,D, p,\omega'_t,I_{t}')$ is a strong $D-$sympelctic homotopy between $\Theta$ and $\Theta^2$
and this will finish the proof.

We begin our construction of $\Phi_t$.
By the one-parameter family version of Moser lemma,
there exist a sufficiently small $\epsilon > 0$ and a smooth family of symplectic embeddings $\Phi=\{ \Phi_{t}\}: (B(\epsilon),\omega_{std}) \to (X,\omega_t)$
sending the origin to $p$ for all $t \in [0,1]$.
Moreover, $\Phi_{0}$ can be chosen to coincide with $I|_{B(\epsilon)}$.
This is not yet the $\Phi_t$ we want.

Notice that $\Phi_1$ is a symplectic embedding of $(B(\epsilon),\omega_{std})$ to $(X,\omega_1)$ sending the origin to $p$
and so is $I^2|_{B(\epsilon)}$.
By possibly choosing a smaller $\epsilon$,
there is a symplectic isotopy of embeddings from $\Phi_1$ to $I^2|_{B(\epsilon)}$ sending the
origin to $p$ for all time, by the trick in Exercise 7.22 of \cite{McSa98}
(This is the trick to prove the space of symplectic embeddings of small balls is connected).
By smoothing the concatenation of $\Phi_t$ with this symplectic isotopy, we can assume that $\Phi_1=I^2|_{B(\epsilon)}$.

We need to further modify $\Phi_t$ by another concatenation. We consider the family of local divisors
Let $F_t=\Phi_t^{-1}(D)$ in the standard coordinates in $(B(\epsilon),\omega_{std})$.
Let $M_t$ be the ordered 2-tuple of the symplectic tangent spaces to the two branches of $F_t$ at the origin.
Since $\Phi_0=I|_{B(\epsilon)}$ and $\Phi_1=I^2|_{B(\epsilon)}$, $M_t$ is a loop.
Let $-M_t$ be the inverse loop of $M_t$ in the space of ordered 2-tuples of positively transversal intersecting two dimensional symplectic vector subspaces.
We can find an isotopy of symplectic embeddings $\Psi_t$ from  $\Phi_1$ to $\Phi_1$ in $(X,\omega_1)$ such that the
corresponding ordered 2-tuple of the symplectic tangent spaces of $\Psi_t^{-1}(D)$ at the origin is $-M_t$.
By concatenating $\Phi_t$ with $\Psi_t$, we can assume at the beginning that the $\Phi_t$ we chose is such that $M_t$ is null-homotopic.
This is the $\Phi_t$ we want which gives a nice family of Darboux balls in $(X,\omega_t)$.

To construct $\omega'_t$, we will isotope the one parameter family of local divisors $F_t$ (fixing both ends) to another one parameter family of symplectic divisors $F_{1, t}$ such that it coincides with $F_0=F_1$
near the origin for all $t$.
First, we perform a one-parameter family of $C^1$ small perturbations to make $F_t$ coincide with a symplectic vector subspace in a sufficiently small ball
$(B(\epsilon_2),\omega_{std})$, where $\epsilon_2 < \epsilon$.
In other words, $F_t$ coincides with $M_t$ in $B(\epsilon_2)$.
Since $M_t$ is null-homotopic, there is a homotopy $W_{r,t}$ between $M_t$ ($r=0$) and the constant path $M_0=M_1$ ($r=1$) such that $W_{r,0}=W_{r,1}=M_0$ for all $r$.
Hence, we can perform a one-parameter family of Lemma 5.10 of \cite{McL12} (See its proof)
to obtain a 3-parameter family of submanifolds $U_{r,s,t}$ in $B(\epsilon_2)$
such that $U_{r,s,t}=W_{s,t}$ outside a fixed small compact set containing the origin, $U_{r,s,t}=W_{r,t}$ close to the origin and $U_{r,r,t}=W_{r,t}$.
As in the proof of Lemma 5.10 of \cite{McL12}, from $U_{r,s,t}$ one can construct
an $s-$parameter  of symplectic isotopy $F_{s,t} \subset B(\epsilon_2)$
such that

$\bullet$ $F_{0,t}=F_t$,

$\bullet$ $F_{s,t}$ is a pair of symplectic submanifolds positively intersecting at the origin for all $s,t \in [0,1]$,

$\bullet$ $F_{1,t}=F_0=F_1=M_0=M_1$ inside $B(\epsilon_4)$ for all $t$,

$\bullet$ $F_{s,0}=F_{s,1}=F_0=F_1$, and

$\bullet$ the isotopy is supported inside $B(\epsilon_3)$, \\
where $0 < \epsilon_4 < \epsilon_3 < \epsilon_2$.

Due to the last bullet, we obtain a $2-$parameter family of marked divisors $D_{s, t}$ with $D_{0, t}=D_t,  D_{s, 0}=D_{s,1}=D$,
and satisfying the  bullets 2 and 3 above  near the marked point (recall we assume there is only one marking for simplicity).

The effect of the symplectic isotopy from $D_t$ ($s=0$) to $D_{1,t}$ ($s=1$) can be converted through symplectomorphism,
as in Lemma \ref{lem: def versus marked def},
to replace $(X,D, p,\omega_t)$ ($s=0$) by an another $D-$symplectic homotopy $(X,D, p,\omega'_t)$ ($s=1$).
More precisely,
 for the 1-parameter family of isotopy $D_{s,t}$ parameterized by $t$,
we can find a 1-parameter family of ambient isotopy $\Delta= \{ \Delta_s \}_{t \in [0,1]}=\{ \Delta_{s,t} \}$,
$\Delta_{s,t}:X \to X$ extending the 1-parameter family of isotopy $D_{s,t}$ (in particular, for fixed $t_0$, $\Delta_{s,t_0}$ is an ambient isotopy extension of $D_{s,t_0}$)
such that
$\Delta_{0,t}=\Delta_{s,0}=\Delta_{s,1}=Id_X$.
Then we define $\omega'_t = \Delta_{1,t}^* \omega_t$.

By construction, we have

$\bullet$ $\omega'_i=\omega_i$ for $i=0,1$,

$\bullet$ $D$ is positively $\omega'_t$-orthogonal for all $t$

$\bullet$ there is a family of symplectic embedding $\Phi_t': B(\epsilon_4) \to (X,\omega_t)$ such that $\Phi_t'^{-1}(D)=F_0$ for all $t$, and

$\bullet$ $\Phi_0'=I|_{B(\epsilon_4)}$ and $\Phi_1'=I^2|_{B(\epsilon_4)}$

In particular, if we let $I_{t}'=\Phi_t'$, then $(X,D, p,\omega'_t,I_{t}')$ is a strong $D-$symplectic homotopy between $\Theta$ and $\Theta^2$.
The strict version follows similarly.
\end{proof}

The ultimate goal for this section is the following proposition, which will be proved after discussing various operations for marked divisors in the next subsection.

\begin{prop} \label{prop: marked blowdown}
Let $\Theta=(X,D,p_j,\omega,I_j)$ and $\Theta^2=(X^2,D^2,p^2_j,\omega^2,I^2_j)$ be two marked divisors both with $l$ marked points.

(i) Up to moving inside the $D-$symplectic deformation class, we
 can blow down a toric or non-toric exceptional class in $\Theta$ (and $\Theta^2$) to obtain a marked divisor $\hat{\Theta}$ (resp. $\hat{\Theta}^2$) with
 an extra marked point (For toric exceptional class, original marked points on the exceptional sphere will be removed after blow-down).

(ii) Moreover, if  the   blow down divisors $\hat{\Theta}$ and $\hat{\Theta}^2$  are $D-$symplectic deformation equivalent such that the extra
marked points correspond to each other in the equivalence, then
$\Theta$ and $\Theta^2$ are $D-$symplectic deformation equivalent.



\end{prop}

\subsection{Operations on marked divisors}

This subsection studies various operations on marked divisors as well as their stability property with respect to $D-$symplectic deformation.

\vspace{1em}
$\bullet$ {\bf Perturbations}
\vspace{1em}

The following fact will be frequently used.

\begin{lemma}\label{lem: perturbation}
Perturbations of a marked divisor preserve the strict $D-$symplectic deformation class.
\end{lemma}

\begin{proof}
A perturbation of a marked divisor is simply a symplectic isotopy of the corresponding (unmarked) divisor. By Lemma \ref{lem: def versus marked def}, the perturbed divisor is symplectomorphic to the original divisor, up to a  $D-$symplectic isotopy.
\end{proof}

\vspace{1em}
$\bullet$ {\bf Marking addition}
\vspace{1em}

A marking addition of a marked divisor $(X,D,\{p_j\}_{j=1}^l,\omega,\{I_j\}_{j=1}^l)$ is another marked
divisor $(X,D,\{p_j\}_{j=1}^{l+1},\omega,\{I_j\}_{j=1}^{l+1})$ with the additional marking $(p_{l+1},I_{l+1})$.

\begin{lemma}\label{lem: marking adding/shrinking}
Let $(X,D,\{p_j\}_{j=1}^l,\omega,\{I_j\}_{j=1}^l)$ be a marked divisor.
If the two marked divisors $(X,D,\{p_j\}_{j=1}^{l}\cup\{q_1\},\omega,\{I_j\}_{j=1}^{l}\cup\{I_{q_1}\})$
together with
$(X,D,\{p_j\}_{j=1}^{l}\cup\{q_2\},\omega,\{I_j\}_{j=0}^{l}\cup \{I_{q_2}\})$ are obtained
by adding markings $(q_{1}, I_{q_1})$ and $(q_2, I_{q_2})$ respectively, then they are strict
$D-$symplectic deformation equivalent if

$\bullet$ the centers $q_1$ and $q_2$ coincide (intersection points of $D$ allowed), or

$\bullet$ $q_1$ and $q_2$ are distinct smooth points of the same irreducible component.

\end{lemma}

\begin{proof}
If $q_1$ and $q_2$ are the same point of $D$, then the claim is trivial since Definition \ref{defn: marked symplectic deformation} only involves the centers of marking, not the coordinates.

If $q_1$ and $q_2$ are smooth points of the same irreducible component, say $C_1$, then  we need to  show that the   4-tuple $(X,D,\{p_j\}_{j=1}^{l}\cup\{q_2\},\omega)$ is  symplectomorphic to a $D-$symplectic isotopy of $(X,D,\{p_j\}_{j=1}^{l}\cup\{q_1\},\omega)$.
For this purpose,
we  find a symplectic isotopy of $(D, \omega|_D)$ fixing $C_1$ setwise, fixing intersection points and $\{p_j\}$ pointwise
and moving $q_1$ to $q_2$.
Using the smooth isotopy extension theorem as in Lemma \ref{lem: def versus marked def},
this isotopy of symplectic divisor gives rise to a smooth isotopy $\Phi_t$ of $X$. The desired  $D-$symplectic isotopy is obtained by taking the $D-$symplectic isotopy to be $(X,D,\{p_j\}_{j=1}^{l}\cup\{q_1\},\Phi_t^*\omega)$ and the symplectomorphism to be $\Phi_1:(X,D,\{p_j\}_{j=1}^{l}\cup\{q_1\},\Phi_1^*\omega)\to (X,D,\{p_j\}_{j=1}^{l}\cup\{q_2\},\omega)$.

\end{proof}

We note that marking addition at an intersection point of a marked divisor is not always possible
because the intersection might not be $\omega$-orthogonal.
However, by Lemma \ref{lem: perturbation}, marking addition at a non-marked intersection point is always possible at the cost of choosing another
representative in the strict $D-$symplectic deformation class because a $C^0$ small perturbation among symplectic
divisor suffices to make the intersection point $\omega$-orthogonal (\cite{Gom95}).

\vspace{1em}
$\bullet$ {\bf Marking moving}
\vspace{1em}

Sometimes, it is useful to be able to move an intersection point.

\begin{lemma}\label{lem: marking moving}
Let $(X,D=C_1\cup C_2 \cup \dots \cup C_k,\{p_j\}_{j=1}^l,\omega,\{I_j\}_{j=1}^l)$ be a marked divisor.
Let $[C_2]^2=-1$ and $p_1=C_1 \cap C_2$.
For any smooth point $\overline{p_1}$ on $C_2$, there is a marked divisor
$(X,\overline{D}=\overline{C_1} \cup C_2 \cup \dots \cup C_k,\{\overline{p_1}\} \cup \{p_j\}_{j=2}^l,\omega',\{\overline{I_j}\}_{j=1}^l)$
such that $\overline{p_1}=\overline{C_1} \cap C_2$, where $\omega'=\omega$ and $C_1=\overline{C_1}$ away from a small open neighborhood of $C_2$.
Moreover, these two marked divisors  are in the same $D-$symplectic deformation equivalence class.
\end{lemma}

\begin{proof}
By Lemma \ref{lem: perturbation}  we may assume that  the intersection points of $D$ are $\omega$-orthogonal.
In particular, if $C_j$ intersects $C_2$, then $C_j$ coincides with a fiber of the symplectic normal bundle of $C_2$ when identifying the symplectic normal bundle with a tubular neighborhood of $C_2$.

Choose an $\omega$-compatible almost complex structure $J$ integrable near $C_2$ which coincides with $(I_j)_*(J_{std})$ for all $j$ and making the symplectic normal bundle a holomorphic vector bundle.
We blows down $C_2$ and identify the ball obtained by blowing down $C_2$ as a chart $(B(\epsilon),\omega_{std},J_{std})$.
In this chart, $C_j$ descends to the union of complex vector subspaces $V_j$ each of which corresponds to an intersection point of $C_2\cap C_j$.
On the other hand, $\overline{p_1}$ being a point on $C_2$ represents a complex vector subspace $V_{\overline{p_1}}$ in this chart.
We take a smooth family of complex vector subspaces $W_t$ from $V_1$ to $V_{\overline{p_1}}$ avoiding $V_j$ for all $j \neq 1$.
Applying the trick in Lemma 5.10 of \cite{McL12} with $N=N'=\emptyset$, $i=1$, $S$ being the center of $B(\epsilon)$, $S_1$ being the descended $C_1$, $W_t=W_1^t$,
we obtain an isotopy of symplectic manifolds $C^t$ supported in $B(\epsilon)$ from the descended $C_1$ (i.e. $C^{t=0}$) to some $C^{t=1}=\tilde{C_1}$ such that $C^t$ coincides with $W_t$ near the origin of $B(\epsilon)$ for all $t$.
By blowing up $B(\epsilon_2) \subset B(\epsilon)$ for some sufficiently small $\epsilon_2$, we can lift this symplectic isotopy to a $D-$symplectic deformation from
$(X,D=C_1\cup C_2 \cup \dots \cup C_k,\{p_j\}_{j=1}^l,\omega,\{I_j\}_{j=1}^l)$ to
$(X,\overline{D}=\overline{C_1} \cup C_2 \cup \dots \cup C_k,\{\overline{p_1}\} \cup \{p_j\}_{j=2}^l,\omega',\{\overline{I_j}\}_{j=1}^l)$ such that
$\overline{p_1}=\overline{C_1} \cap C_2$, where $\overline{C_1}$ is the proper transform of $\tilde{C_1}$.
\end{proof}

\vspace{1em}
$\bullet$ {\bf Canonical blow-up}
\vspace{1em}


Given a marked divisor with $l$ markings, there are $l$ canonical blow-ups we can do, namely, blow-ups using the symplectic
embeddings $I_j$ and hence the blow-up size is $B(\delta_j)$.
A canonical blow-up of a marked divisor is still a marked divisor with one less
the number of $p_j$'s.


\begin{lemma}\label{lem: canonical blow-ups}
 If $\Theta=(X,D,\{p_j\}_{j=1}^l,\omega,\{I_j\}_{j=1}^l)$ and
$\Theta^2=(X^2,D^2,\{p^2_j\}_{j=1}^l,\omega^2,\{I^2_j\}_{j=1}^l)$ are $D-$symplectic deformation equivalent,
then so are the marked divisors obtained by canonical blow-ups using $I_1$ and $I^2_1$.
\end{lemma}

\begin{proof}
 By Lemma \ref{lem: D-deformation=strong D-deformation}, $\Theta$ and $\Theta^2$ are strong $D-$symplectic deformation equivalent.
 By blowing up using $I_{1,t}$, we obtain a $D-$symplectic deformation equivalence between the blown-up marked divisors.
\end{proof}

\subsection{Proof of Proposition \ref{prop: marked blowdown}}




\begin{proof}[Proof of Proposition \ref{prop: marked blowdown}]


For a non-toric class $e$, we can find by Lemma \ref{lem: existence of nice J-sphere},
a pseudo-holomorphic representative $E$ such that $D$ is at the same time pseudo-holomorphic, after possibly applying Lemma \ref{lem: perturbation}
to move $\Theta$ in the strict $D-$symplectic deformation class.
By positivity of intersection, $E$ intersects exactly one irreducible component
of $D$ and the intersections is positively transversally once and hence a non-toric exceptional curve.
By perturbing $E$, we can assume $E$ has $\omega$-orthogonal intersection with $D$.
We can get a marked divisor after blowing down $E$ with a marked point corresponds to the contracted $E$.

For a toric class $e$, we again apply Lemma \ref{lem: perturbation} to move $\Theta$ in its strict $D-$symplectic deformation class such that
every intersection is $\omega$-orthogonal.
The irreducible component $E$ of $D$ in the class $e$ is a toric exceptional sphere.
Hence, $E$ intersects two other irreducible components of $D$ once.
We apply Lemma \ref{lem: marking moving} to find another representative of $\Theta$ in the $D-$symplectic deformation class
such that after we blow down the exceptional curve, the intersection point corresponding to the exceptional curve is an $\omega$-orthogonal
intersection point so this descended divisor is still a marked divisor
(recall, a marking for a marked divisor at an intersection point requires the intersection point is an $\omega$-orthogonal intersection).


Finally, suppose the blow down divisors   are $D-$symplectic deformation equivalent.
We want to do canonical blow-ups and marking additions to recover our original divisor $D$ and $D^2$.
Notice that, marking additions are needed because when one blow down a divisor which originally has markings on it, the marking will not persist after the blow-down. Therefore, when we blow up the symplectic ball back, we need marking additions to get back the original marked divisor.
We remark that we may not get back exactly the pair of $D$ and $D^2$ by just canonical blow-ups and marking additions but we can get some pair   in the same
$D-$symplectic deformation equivalence classes  by Lemma \ref{lem: perturbation}.


Since $D-$symplectic deformation equivalence is stable under canonical blow-ups (Lemma \ref{lem: canonical blow-ups}) and marking additions (Lemma \ref{lem: marking adding/shrinking}), we
conclude that $\Theta$ is $D-$symplectic deformation equivalent to $\Theta^2$.

\end{proof}


\section{Minimal models}\label{section: Reduction to minimal models}

We first collect some  facts, which should be well known to experts.

\begin{lemma} \label{lemma: genus 1}
Let $(X, D, \omega)$ be a symplectic log Calabi-Yau surface. Then $X$ is rational or an  elliptic ruled surface, and
$D$ is either a torus or a cycle of spheres. If $(X, D, \omega)$ is a symplectic Looijenga pair, then $(X, \omega)$
is rational.
\end{lemma}

\begin{proof}
Since $D$ is symplectic and $[D]=PD(c_1(X, \omega))$, we have $c_1(X, \omega) \cdot [\omega]=[D] \cdot [\omega] >0$.
 By Theorem A of \cite{Liu96} or \cite{OhOn96}, $X$ is rational or ruled.

 Write $D=C_1 \cup C_2 \dots \cup C_k$, where each $C_i$ is a smoothly embedded closed symplectic genus $g_i$ surface.
 By adjunction, we have $[C_i]\cdot[D]=[C_i]^2+2-2g_i$.
 Therefore, we have $$[C_i]\cdot (\sum\limits_{j \neq i}[C_j])=2-2g_i \ge 0.$$
 In particular, we have $g_i \le 1$ for all $i$.
 Since we assumed $D$ is connected (we always assume a symplectic divisor is connected),  $D$ is either a torus or a cycle of spheres.
 Here a cycle of spheres means that the dual graph is a circle and each vertex has genus $0$.

If $X$ is not rational, then  $X$ admits an $S^2-$fibration structure over a Riemann surface of positive genus.  After possibly smoothing,
we get a torus  $T$ representing the class $c_1(X)$.
 Moreover, $c_1(X)(f)=2$ where $f$ is the fiber class. The projection from $T$ to the base
 is of non-zero degree.
 Therefore, the base genus of $X$ is at most $1$.

 If $(X,D,\omega)$ is a symplectic Looijenga pair, then at least one of the sphere component pairs positively with the fiber class (by $c_1(X)(f)=2$ again).
 Hence, the base genus is $0$ and $X$ is rational.
\end{proof}


 For a cycle with $k$ spheres we will also call it a $k-$gon, and a torus a $1-$gon.
 If we allow some $C_i$ to be positively immersed, then
 by adjunction we see that the only possibility is a single sphere with one positive double point,
 which we call a degenerated 1-gon.

The following observations are straightforward.

\begin{lemma}\label{lemma: same type}
The operations of toric blow-up, non-toric blow-up,
toric blow-down and non-toric blow-down   all  preserve being symplectic log Calabi-Yau.

\end{lemma}

In the next subsection it is convenient to apply a slightly more general version of toric blow-down:
Suppose a component $C$  of a bi-gon $D$  is an exceptional sphere.
The generalized  toric blow down of $D$ along $C$  is blowing down $C$, which results in a degenerated 1-gon.
Notice that the homology class of a degenerated 1-gon is still Poincare dual to the first Chern class.




\subsection{Minimal reductions}

\begin{defn}\label{defn: minimal model}
A  symplectic log Calabi-Yau surface  $(X,D,\omega)$ is called
 a {\bf minimal model}  if either $(X, \omega)$ is minimal,  or $(X,D,\omega)$  is a symplectic  Looijenga pair
with $X=\mathbb{C}P^2 \# \overline{\mathbb{C}P^2}$.

\end{defn}

\begin{lemma}\label{lemma: fundamental theorem + minimal model}
Every symplectic log Calabi-Yau surface  can be transformed   to a minimal model via a sequence of non-toric blow-downs followed by a sequence of toric blow-downs.
\end{lemma}

\begin{proof}

{\bf  Non-toric blow-down}
  Suppose $e$ is an exceptional class intersecting each component of $D$ non-negatively.
  Then $e$ is a non-toric exceptional class by adjunction.

By Lemma \ref{lem: existence of nice J-sphere}, there is an $\omega$-compatible almost complex structure such that
$D$ $J$-holomorphic (possibly after perturbation of $D$) and $e$ has an embedded $J$-holomorphic sphere representative $E$.
Thus we can perform non-toric blow-down along $E$.

By iterative non-toric blow-downs, we  end up with a symplectic log Calabi-Yau surface $(X_0,D_0,\omega_0)$ such that any exceptional class pairs negatively with some component of $D$.

{\bf Toric blow-down}

If $X_0$ is not minimal and not diffeomorphic to $\mathbb{C}P^2 \# \overline{\mathbb{C}P^2}$, then for any $\omega_0$-compatible $J_0$ making $D_0$ $J_0$-holomorphic, the exceptional class
with minimal $\omega_0$-area has an embedded $J_0$-holomorphic representative, by Lemma 1.2 of \cite{Pi08}.
Therefore, this embedded representative must coincide with an irreducible component $C$ of $D_0$.

Therefore if $D_0$ is a torus then $X_0$ must be minimal. So from now on we assume that
$D_0$ is a cycle of spheres, ie. $(X_0,D_0,\omega_0)$ is a Looijenga pair.


Suppose that  $C$ intersects two other components of $D_0$ and hence a toric exceptional sphere.
In this case we perform  toric blow down along $C$ to  get another  symplectic Looijenga pair $(X_0',D_0',\omega_0')$.
 We claim that there is no exceptional class in $X_0'$ that pairs all irreducible components of $D_0'$ non-negatively.
If there were one, by Lemma \ref{lem: existence of nice J-sphere}, after possibly perturbing $D_0'$ to be $\omega_0'-$orthogonal, then there would be a embedded pseudo-holomorphic representative $E_0'$ intersecting exactly one irreducible component of $D_0'$ transversally at a smooth point.
This $E_0'$ can be lifted to the symplectic log Calabi-Yau surface $(X_0,D_0,\omega_0)$ because the contraction of $C$ becomes an intersection point of $D_0'$, which is away from $E_0'$.
Contradiction.
Therefore, we can continue to perform toric blow-down until the ambient manifold is minimal, diffeomorphic to $\mathbb{C}P^2\# \overline{\mathbb{C}P^2}$ or the minimal area exceptional sphere
intersect only one irreducible component of the divisor.

We now consider the case that the minimal area expectional sphere $C$ only intersects
with one component of the divisor $D_0$, then $D_0$ must be  a bigon.
We claim that
 $X_0=\mathbb{C}
P^2 \# \overline{\mathbb{C}P^2}$ in this case, and hence $(X_0,D_0,\omega_0)$ is minimal,
according to Definition \ref{defn: minimal model}.
To see why $X_0=\mathbb{C}
P^2 \# \overline{\mathbb{C}P^2}$, we apply a generalized toric blow-down along $C$ to obtain $(X_0',D_0',\omega_0')$
where $D_0'$ is a degenerated 1-gon.
We next show that $(X_0',  \omega_0')$ is minimal.
After possibly perturbing  the nodal point of $D_0'$ to be
$\omega_0'-$orthogonal so $D_0'$ can be made a pseudo-holomorphic nodal sphere, the analysis above also shows that there is no exceptional class in $X_0'$ that intersects $[D_0']$ non-negatively.
Since $D_0'$ represents the Poinc\'are dual of $c_1(X_0', \omega_0')$, there are also no
exceptional class intersecting $[D_0']$ negatively.
Thus, it means that $X_0'=\mathbb{C}P^2$ or $\mathbb{S}^2\times \mathbb{S}^2$.
If $X_0'$ is $\mathbb{S}^2 \times \mathbb{S}^2$, then $D_0'$ is obtained by blowing down a component of a bi-gon  $D_0$ in $X_0=\mathbb{C}P^2 \# 2\overline{\mathbb{C}P^2}$.
In this case there are three exceptional class in $(X_0, \omega_0)$ with pairwise intersecting number 1.
It is simple to check  by adjunction that any  exceptional class not represented by
  any of the two  components of $D_0$ is non-toric. But  this situation would not appear due to our procedure which performs non-toric blow down first.
Hence the only possibility is that $X_0'=\mathbb{C}P^2$, from which it follows that
$X_0=\mathbb{C}
P^2 \# \overline{\mathbb{C}P^2}$.

In summary, we can do iterative toric blow-downs from $(X_0,D_0,\omega_0)$ to obtain a
symplectic Looijenga pair $(X_b,D_b,\omega_b)$ such that  either $(X_b, \omega_b)$ is minimal or $X_b$ is diffeomorphic to $\mathbb{C}P^2 \# \overline{\mathbb{C}P^2}$.

 \end{proof}



From Lemma \ref{lemma: genus 1}, Lemma \ref{lemma: same type}, Lemma  \ref{lemma: fundamental theorem + minimal model} and adjunction formula,
we can enumerate the minimal symplectic log Calabi-Yau surfaces up to the homology of the irreducible components.



$\bullet$  Case $(A)$: The base genus of $X$ is $1$. $D$ is a torus.





$\bullet$  Case $(B)$: $X=\mathbb{C}P^2$. $c_1=3H$. Then the symplectic log Calabi-Yau are

$(B1)$ $D$ is a torus,

$(B2)$ $D$ consists of a $H-$sphere and a $2H-$sphere, or

$(B3)$ $D$ consists of three $H-$sphere.

$\bullet$ Case $(C)$: $X=\mathbb{S}^2 \times \mathbb{S}^2$, $c_1=2f+2s$, where $f$ and $s$ are homology class
of the two factors. By adjunction, the homology $af+bs$ of any embedded symplectic sphere satisfies $a=1$ or $b=1$.
Symplectic log Calabi-Yau surfaces are

$(C1)$ $D$ is a torus.

$(C2)$ If $D$ has two irreducible components $C_1$ and $C_2$, then the only possible case (modulo obvious symmetry) is $[C_1]=f+bs$ and $[C_2]=f+(2-b)s$.
Its graph  is given by
$$    \xymatrix{
       \bullet^{2b} \ar@{=}[r] & \bullet^{4-2b}
	}
$$


$(C3)$ If $D$ has three irreducible components $C_1$, $C_2$ and $C_3$, then the only possible case (modulo obvious symmetry) is $[C_1]=f+bs$, $[C_2]=f+(1-b)s$ and $[C_3]=s$.
Its graph is given by
$$    
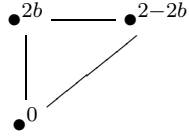
\xymatrix{
       \bullet^{2b} \ar@{-}[d] \ar@{-}[r] & \bullet^{2-2b} \ar@{-}[dl] \\
       \bullet^{0}
	}
$$

$(C4)$ If $D$ has four irreducible components, then the only possible case (modulo obvious symmetry) is $[C_1]=f-bs$, $[C_2]=f+bs$, $[C_3]=s$ and $[C_3]=s$.
Its graph is given by
$$    
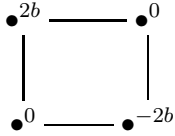
\xymatrix{
       \bullet^{2b} \ar@{-}[d] \ar@{-}[r] & \bullet^{0} \ar@{-}[d] \\
       \bullet^{0} \ar@{-}[r] & \bullet^{-2b} \\
	}
$$

It is not hard to draw contradiction if $D$ has $5$ or more irreducible components.



$\bullet$ Case $(D)$: $X=\mathbb{C}P^2 \# \overline{\mathbb{C}P^2}$. $c_1=f+2s$, where $f$ and $s$ are fiber class and section class, respectively,
such that $f^2=0$, $f\cdot s=1$ and $s^2=1$.
By adjunction, the homology $af+bs$ of an embedded symplectic sphere satisfies $b=1$ or $b=2-2a$.

$(D1)$ $D$ cannot be a torus because it would not be minimal.

$(D2)$ If $D$ has two irreducible components $C_1$ and $C_2$, then the only two possible cases (modulo obvious symmetry) are
$([C_1],[C_2])=(af+s,(1-a)f+s)$ and $([C_1],[C_2])=(f,2s)$.
The graphs are given by
$$    \xymatrix{
       \bullet^{2a+1} \ar@{=}[r] & \bullet^{3-2a}
	}
$$
and
$$    \xymatrix{
      \bullet^{4} \ar@{=}[r] & \bullet^{0}
	}
$$

$(D3)$  If $D$ has three irreducible components, then the only possible case (modulo obvious symmetry) is $[C_1]=af+s$, $[C_2]=-af+s$ and $[C_3]=f$.
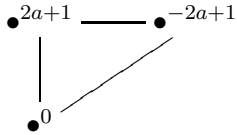
$$    \xymatrix{
       \bullet^{2a+1} \ar@{-}[d] \ar@{-}[r] & \bullet^{-2a+1} \ar@{-}[dl] \\
       \bullet^{0}  \\
	}
$$



$(D4)$ If $D$ has four irreducible components, then the only possible case (modulo obvious symmetry) is $[C_1]=af+s$, $[C_2]=-(a+1)f+s$, $[C_3]=f$ and $[C_4]=f$.
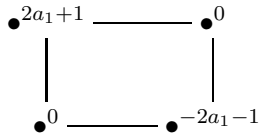
$$    \xymatrix{
       \bullet^{2a_1+1} \ar@{-}[d] \ar@{-}[r] & \bullet^{0} \ar@{-}[d] \\
       \bullet^{0} \ar@{-}[r] & \bullet^{-2a_1-1} \\
	}
$$

It is not hard to draw contradiction if $D$ has $5$ or more irreducible components.

\subsection{Deformation classes of minimal models}\label{section: Deformation Problem for minimal models}

In this section, we study the symplectic deformation classes of minimal symplectic log Calabi-Yau surfaces.


\begin{prop}\label{prop: isotopy class in minimal model}
Let $(X,D=C_1 \cup \dots \cup C_k,\omega)$ be a minimal symplectic log Calabi-Yau surface.
If $\overline{D}=\overline{C_1} \cup \dots \cup \overline{C_k} \subset (X,\omega)$ is another symplectic divisor representing the first Chern class such that $[C_i]=[\overline{C_i}]$ for all $i$.
Then $(X,D,\omega)$ is symplectic deformation equivalent to $(X,\overline{D},\omega)$.
\end{prop}

The proof of Proposition \ref{prop: isotopy class in minimal model} is separated into two cases,
Proposition \ref{prop: Proof for symplectic Looijenga pairs} and Proposition \ref{prop: torus type}.

\subsubsection{Isotopy in rational surfaces}


\begin{prop}\label{prop: Proof for symplectic Looijenga pairs}
Suppose $(X, D, \omega)$ and $(X, \overline{D}, \omega)$ satisfy the assumtion of Proposition \ref{prop: isotopy class in minimal model} such that, in addition, $X$ is rational, then
$D$ is symplectic isotopic to $\overline{D}$.

\end{prop}

The proof of Proposition \ref{prop: Proof for symplectic Looijenga pairs} when $D$ is a torus is given by \cite{Si03} and Theorem B and Theorem C of \cite{SiTi05}.
We only need to  deal with symplectic Looijenga pairs.



Recall that cohomologous symplectic forms on a rational or ruled 4-manifold
are symplectomorphic (cf. \cite{Ta95}, \cite{LaMc96-2}  and the survey \cite{Sa13}).
Therefore it suffices to consider   the following 'standard symplectic models' for  $\mathbb{S}^2 \times \mathbb{S}^2$,  $\mathbb{C}P^2$  and $\mathbb{C}P^2 \# \overline{\mathbb{C}P^2}$.






$\bullet$ $\mathbb{S}^2 \times \mathbb{S}^2$ model:

When $X$ is diffeomorphic to $\mathbb{S}^2 \times \mathbb{S}^2$, we define the product symplectic form $\omega_{\lambda}=(1+\lambda)\sigma \times \sigma$ with $\sigma$ a symplectic form on the second factor with area $1$ and $\lambda \ge 0$.
Let $E_0$ be the class of the first factor, $F$ be the class of the second factor and $E_{2k}=E_0-kF$ for $0 \le k \le l$, where $l$ is the integer with $l-1 < \lambda \le l$.
For $0 \le k \le l$, let $U_k$ be the set of $\omega_{\lambda}$-compatible almost complex structure such that $E_{2k}$ is represented by an embedded pseudo-holomorphic sphere.

$\bullet$ $\mathbb{C}P^2$ model:

When $X$ is diffeomorphic to $\mathbb{C}P^2 $, we use   a multiple of the Fubini-Study form,  $c\omega_{FS}$.

$\bullet$ $\mathbb{C}P^2 \# \overline{\mathbb{C}P^2}$ model:

When $X$ is diffeomorphic to $\mathbb{C}P^2 \# \overline{\mathbb{C}P^2}$, we use $\omega_{\lambda}$ to denote a form obtained by blowing up $(\mathbb{C}P^2, (2+\lambda)\omega_{FS})$ with size $1+\lambda$.
So the line class $H$ has area $2 +\lambda$ and the exceptional class $E_1$ has area $1+\lambda$, where $\lambda > -1$.
Let $F=H-E_1$ be the fiber class and let also $E_{2k+1}=E_1-kF$ for $0 \le k \le l$, where $l$ is again the integer with $l-1 < \lambda \le l$.
Similarly,  let $U_k$ be the space of $\omega_{\lambda}$-compatible almost complex structure such that $E_{2k+1}$ is represented by an embedded pseudo-holomorphic sphere.

\begin{prop}\label{prop: M-O}(Proposition 2.3 and Corollary 2.8 of \cite{AbMc00}, see also Proposition 6.4 of \cite{LiWu12})
Let $(X,\omega_{\lambda})$ be one of the above two cases.
For each $0\le k \le l$, $U_k$ is non-empty and path connected.
As a result, any two embedded symplectic spheres $C_0$ and $C_1$ representing the same class $E_j$ for some $0\le j \le 2l+1$ are symplectic isotopic to each other.
\end{prop}

\begin{lemma}\label{lem: negative curve in sphere bundle over sphere completely coincide}
Let $(X,\omega_{\lambda})$ be as in Proposition \ref{prop: M-O}.
Assume $C_0,C_1 \subset X$ are two embedded symplectic spheres representing the same class $E_j$ for some $0\le j \le 2l+1$.
Then there is a Hamiltonian diffeomorphism of $(X,\omega_{\lambda})$ sending $C_0$ to $C_1$.
\end{lemma}

\begin{proof}
By Proposition \ref{prop: M-O}, we can find a symplectic isotopy $C_t \subset X$ from $C_0$ to $C_1$.
We can extend this symplectic isotopy from a neighborhood of $C_0$ to a neighborhood of $C_1$ by a Moser type argument(See e.g. Chapter $3$ of \cite{McSa98}).
Our aim is to extend this symplectic isotopy to an ambient symplectic isotopy in order to obtain the result.

We first extend this symplectic isotopy to an ambient diffeomorphic isotopy $\Phi: X \times [0,1] \to X$.
By considering the pull-back form $\Phi^*\omega_{\lambda}$, we can identify $C_0=\Phi_t^{-1}(C_t)$ for all $t$ in the family of symplectic manifold
$(X \times \{t\}, \Phi^*\omega_{\lambda}|_{X \times \{t\}})$, as in Lemma \ref{lem: def versus marked def}.
We denote $\Phi^*\omega_{\lambda}|_{X \times \{t\}}$ as $\omega_{\lambda}^t$.
By definition, $\omega_{\lambda}^t$ is fixed near $C_0$ for all $t$.
Identify a tubular neighborhood of $C_0$ with a symplectic normal bundle.
Then, choose a smooth family of $\omega_{\lambda}^t$-compatible almost complex structure $J_t$ on $X$ such that $J_t$ is fixed near $C_0$ and the
fibers of the normal bundle of $C_0$ are $J_t$-holomorphic.
Pick a point $p_0$ on $C_0$.
Let the $J_t$ holomorphic sphere representing the fiber class $F$ and passing through $p_0$ be $C^F_t$.
Since the fiber class with a single point constraint has Gromov-Witten invariant one or minus one, $C^F_t$ forms a symplectic isotopy by Gromov compactness.
By Lemma 3.2.1 of \cite{McOp13} (let $C_0$ be $C^{S_1}$ and $[C^F_t]$ be $B_1$), we can
assume that the intersection between $C_0$ and $C^F_t$ is $\omega_{\lambda}^t$-orthogonal, after possibly perturbing $J_t$.

Now, $\Phi(C_0,t) \cup \Phi(C^F_t,t)=C_t \cup \Phi(C^F_t,t)$ is an $\omega_{\lambda}$ orthogonal symplectic isotopy
in $(X,\omega_{\lambda})$ (Strictly speaking, $C^F_t$ is the image of another diffeomorphic isotopy $\Psi$ such that
$C^F_t=\Psi(C^F_0,t)$ and $C_0=\Psi(C_0,t)$, then the isotopy we want is $\Phi(\Psi(\cdot,t),t)$).
We can extend this symplectic isotopy to a neighborhood of it by another Moser type argument since $\Phi(C_0,t)$ intersects $\Phi(C^F_t,t)$ $\omega_{\lambda}$-orthogonally.
We have the exact sequence
$$H^1(C_0 \cup C^F_0,\mathbb{R})=0 \to H^2(X,C_0 \cup C^F_0,\mathbb{R}) \to H^2(X,\mathbb{R}) \to H^2(C_0 \cup C^F_0,\mathbb{R})$$
where the last arrow is an isomorphism and hence $H^2(X,C_0 \cup C^F_0,\mathbb{R})=0$.
By Banyaga extension theorem (See e.g. \cite{McSa98}), there is an ambient symplectic isotopy agree with the symplectic isotopy $C_t \cup \Phi(C^F_t,t)$.
Finally, this ambient symplectic isotopy is a Hamiltonian isotopy because $H^1(X)=0$.
\end{proof}

\begin{proof} [Proof of Proposition \ref{prop: Proof for symplectic Looijenga pairs}]

As seen in the previous section, $D$ and $\overline{D}$ have at most four irreducible components.
We are going to prove Proposition \ref{prop: Proof for symplectic Looijenga pairs} by dividing it into the cases of two, three or four irreducible components.
The proof for bigons is written with details,
while the proof for triangles or rectangles being similar to that of bigons will be sketched.

$\bullet$ Bigons

First, let $(X,\omega)=(\mathbb{S}^2 \times \mathbb{S}^2, c\omega_{\lambda})$ for some constant $c$, $D=C_1 \cup C_2$, $\overline{D}=\overline{C_1} \cup \overline{C_2}$ and
$[C_i]=[\overline{C_i}]$ for $i=1,2$.
Without loss of generality, we may assume $[C_1]^2 \le [C_2]^2$.
From the enumeration, we have $[C_1]=F+(2-b_1)E_0$ and $[C_2]=F+b_1E_0$ for some $b_1 \ge 1$, or $[C_1]=(2-a_1)F+E_0$ and $[C_2]=a_1F+E_0$ for some $a_1 \ge 1$.
We consider the latter case and the first case can be treated similarly.

We first consider $a_1 \ge 2$.
By Lemma \ref{lem: negative curve in sphere bundle over sphere completely coincide},
after composing a Hamiltonian diffeomorphism, we can assume $C_1$ and $\overline{C_1}$ completely coincide.
Fix an $\omega$-tamed almost complex structure $J_0$ making $C_1=\overline{C_1}$ pseudo-holomorphic and integrable near $C_1$.
Consider the set of $\omega$-tamed almost complex structure $\mathcal{J}$ agree with $J_0$ near $C_1$.
Fix $J\in \mathcal{J}$, we want to inspect all possible degenerations of $J$-holomorphic nodal curve representing $[C_2]$.
By positivity of intersection, positivity of area and adjunction, the homology class $aF+bE_0$ of any $J$-holomorphic curve has
non-negative coefficient for the $E_0$ factor (i.e. $b\ge 0$).
Therefore, the irreducible components (possibly not simple) of any $J$-holomorphic curve representing $[C_2]$ give rise to a decomposition
$[C_2]=(s_1F+E_0)+s_2F+\dots+s_mF$, where $s_j > 0$ for $2 \le j \le m$
(by positivity of intersection with $[C_1]$).
If $s_1 \le 0$, then $s_1F+E_0=[C_1]$ by positivity of intersection with $[C_1]$.
The sum of non-negative Fredholm index of the underlying curve of each individual component is given by
$Ind_{nodal}=(4s_1+2)+2(m-1)$
when $s_1 \ge 0$, and
$Ind_{nodal}=2(m-1)$
when $s_1 < 0$
because the class $s_1F+E_0$ is primitive and the underlying curve for $s_jF$ has homology $F$ (the index formula for a pseudo-holomorphic curve with class $A$ is $2c_1(A)-2$).
On the other hand, the index of the class $[C_2]$ is given by
$Ind_{C_2}=2(2a_1+2)-2=4(\sum_{i=1}^m s_i)+2 = (4s_1+2)+4(\sum_{i=2}^m s_i)$.
If $s_1 \ge 0$ and $m\ge 2$, we have
$$Ind_{nodal}+2 \le (4s_1+2)+4(\sum\limits_{i=2}^m s_i) = Ind_{C_2}$$
If $s_1 < 0$, we have $s_1=2-a_1$ and hence
$$Ind_{nodal}+2 = 2(m-1)+2 \le 2(\sum\limits_{i=2}^m s_i)+2=2(a_1-(2-a_1))+2=4a_1-2 <Ind_{C_2}$$
Therefore, any degeneration happens in codimension two or higher.

Then we can apply the standard pseudo-holomorphic curve argument to obtain a symplectic isotopy from $C_2$ to $\overline{C_2}$
transversal to $C_1$ for all time along the isotopy and finish the proof.
Since we could not find references that fit exactly to out situation (Proposition 1.2.9(ii) of \cite{McOp13} is a very closely related one), we provide some details here.
We will basically follow \cite{McSa12} together with Lemma 3.2.2 and Proposition 3.2.3 of \cite{McOp13}.

We perturb $C_2$ and $\overline{C_2}$ so that they have $2a_1+1$ distinct intersection points and call these intersection points $\{p_j\}_{j=1}^{2a_1+1}$.
We form the universal moduli space for genus $0$ curve representing the class $[C_2]$  with $2a_1+1$
point constraints  $\{p_j\}_{j=1}^{2a_1+1}$ with respect to the space of almost complex structures $\mathcal{J}$.
We want to pick $J,\overline{J} \in \mathcal{J}$ that are regular for all underlying (marked) simple curves
appearing in a degeneration of $[C_2]$ except $C_1=\overline{C_1}$ such that $C_2$ is $J$-holomorphic and $\overline{C_2}$ is $\overline{J}$-holomorphic.

To find $J$ and $\overline{J}$, we note the following two facts.
For any $J \in \mathcal{J}$ (resp. $\overline{J} \in \mathcal{J}$) making $C_2$ $J$-holomorphic
(resp. making $\overline{C_2}$ $\overline{J}$-holomorphic), the Fredholm operator taking the point constraints
$\{p_j\}_{j=1}^{2a_1+1}$ into account is regular by automatic transversality (See Theorem 3.1 and Proposition 3.2 of \cite{LiWu12}, and also \cite{HoLiSi97}, \cite{IvSh99}).
On the other hand, for a generic choice of
$J$ (resp. $\overline{J}$) making $C_1$ and $C_2$ $J$-holomorphic (resp. $C_1=\overline{C_1}$ and
$\overline{C_2}$ $\overline{J}$-holomorphic), each simple curve
other than $C_1$ and $C_2$ (resp. other than $C_1$ and $\overline{C_2}$) in any degeneration has a somewhere injective point away from $C_1$ and $C_2$
(resp. away from $\overline{C_1}$ and $\overline{C_2}$)
and hence is regular (See Chapter 3.4 of \cite{McSa12}).
As a result, we can find  $J,\overline{J} \in \mathcal{J}$ as desired.

For such $J,\overline{J}$, there is a regular smooth path $J_t \in \mathcal{J}$ (regular in the sense of Definition 6.2.10 of \cite{McSa12}) such that the parametrized moduli space of
$J_t-$holomorphic curves representing $[C_2]$
and passing through  $\{p_j\}_{j=1}^{2a_1+1}$ forms a non-empty one dimensional smooth manifold.
Since degeneration happens in codimension $2$ or higher, if we choose $J_t$ to be also regular
with respect to the lower strata, the one dimensional moduli space is also compact.

Thus, there is a family of embedded $J_t$-holomorphic spheres $C^t$ all of which passing through $\{p_j\}_{j=1}^{2a_1+1}$.
By positivity of intersection, $C^t$ is the only $J_t$-holomorphic family passing through $\{p_j\}_{j=1}^{2a_1+1}$, hence
we have a symplectic isotopy from $C_2$ to $\overline{C_2}$.
Finally, by applying Lemma 3.2.1 of \cite{McOp13} to $\{C^t\}$
to get another symplectic isotopy $\{C^{t'}\}$ transversal to $C_1$, we get that the intersection pattern of $\{C^{t'}\} \cup C_1$ is unchanged  along the symplectic isotopy.
This finishes the proof  when $a_1 \ge 2$.

The case that $a_1=1$ can be treated similarly, which is easier and only requires an analogue of
Proposition \ref{prop: M-O} and Lemma \ref{lem: negative curve in sphere bundle over sphere completely coincide} for symplectic sphere with non-negative self-intersection (See e.g Proposition 3.2 of \cite{LiWu12}).

Now, we consider $(X,\omega)=(\mathbb{C}P^2 \# \overline{\mathbb{C}P^2}, c\omega_{\lambda})$ for some constant $c$,
$D=C_1 \cup C_2$, $\overline{D}=\overline{C_1} \cup \overline{C_2}$ and
$[C_i]=[\overline{C_i}]$ for $i=1,2$.
By the enumeration, there are two possible cases.

The first one is when
$[C_1]=[\overline{C_1}]=(1-a_1)f+s=(2-a_1)F+E_{1}$
and $[C_2]=[\overline{C_2}]=a_1f+s=(a_1+1)F+E_{1}$.
By symmetry, it suffices to consider $a_1 \ge 1$.
If $a_1 \ge 2$, we apply Lemma \ref{lem: negative curve in sphere bundle over sphere completely coincide} and assume $C_1$ completely coincides with $\overline{C_1}$.
Again, we inspect all possible $J$-holomorphic degenerations of $C_2$ for $J$ making $C_1$ $J$-holomorphic.
A direct index count as before shows that any degeneration of $C_2$ has at least codimension two.
Therefore, the same method applies.
The case that $a_1=1$ is dealt similarly.

The other case is $[C_1]=[\overline{C_1}]=f=F$ and
$[C_2]=[\overline{C_2}]=2s=2F+2E_{1}$.
This cannot cause additional trouble as they have non-negative self-intersection numbers.
One can deal with this similar to the previous cases.

The case that $X=\mathbb{C}P^2$ is analogous and easier.

$\bullet$ Triangles and Rectangles

Now, we consider $X=\mathbb{S}^2 \times \mathbb{S}^2$ or
$X=\mathbb{C}P^2 \# \overline{\mathbb{C}P^2}$ and assume $D,\overline{D}$ has three or four irreducible components.
We observe that, there is at most one component with negative self-intersection number and one with positive self-intersection numbers in all cases.
Moreover, the homology class of the component with negative self-intersection number is of the form
$E_{i}+jF$ for some $j$ and $i=0,-1$.
If there is a negative self-intersection component, we can apply Lemma \ref{lem: negative curve in sphere bundle over sphere completely coincide} and
assume the negative self-intersection components for $D$ and $\overline{D}$ completely coincide.
Then we study all the possible $J$-holomorphic degeneration of the {\bf positive} curve for $J$ making the negative component $J$-holomorphic.
One can show that the degeneration happens in at least codimension two by index count.
Therefore, we can find a relative pseudo-holomorphic isotopy $\Phi_t$ from the positive self-intersection component of $D$ to the positive self-intersection component of $\overline{D}$.
At the same time, since the remaining components of $D$ and $\overline{D}$ are sphere fibers, which cannot have any pseudo-holomorphic degeneration,
the pseudo-holomorphic isotopy $\Phi_t$ can be extended to a pseudo-holomorphic isotopy from $D$ to $\overline{D}$.
Hence, the result follows when there is a negative self-intersection component.
The remaining cases are all similar and simpler, including the case when $X=\mathbb{C}P^2$.
\end{proof}

\subsubsection{Elliptic ruled surfaces}

In this subsection, we want to finish the proof of Proposition \ref{prop: isotopy class in minimal model} for the torus type.

\begin{prop}\label{prop: torus type}
Suppose $(X, D, \omega)$ and $(X, \overline D, \omega)$ are minimal symplectic log Calabi-Yau surfaces such that $X$ is elliptic ruled. Then they are symplectic deformation equivalent to each other.
\end{prop}



We first describe the complement of $D$ following \cite{Us09}.
Any $\omega$-compatible almost complex structure $J$ provides us a $J$-holomorphic ruling, meaning that there is a sphere bundle map $\pi: X \to \mathbb{T}^2$ such that fibers are $J$-holomorphic.
Usher proves in \cite{Us09} (Lemma 3.5) that, if $D$ is $J$-holomorphic, $\pi|_{D}$ is a two to one covering and in particular $D$ is transversal to the $J$-holomorphic sphere foliation.
If a tubular neighborhood of $D$ is taken out, we have a $J$-holomorphic annulus foliation, which defines an annulus bundle $X-P(D)$ over the torus $\mathbb{T}^2$.
We want to identify this annulus bundle.

Equip the orientation of $\mathbb{T}^2$ such that $\pi|_{D}$ is orientation preserving, where the orientation of $D$ is determined by $J$.
Choose a positively oriented basis $\{t,u\} \in H_1(D,\mathbb{Z})$ and $\{v,w\} \in H_1(\mathbb{T}^2,\mathbb{Z})$ such that $\pi_*t=v$ and $\pi_*u=2w$.
Let $\mathbb{A}=\{z \in \mathbb{C}| \frac{1}{2} \le |z| \le 2 \}$.
The monodromy of this annulus bundle around the loop corresponding to $v$ is orientation preserving and does not flip the boundary.
Therefore, the monodromy is isotopic to the identity.
Similarly, the monodromy of this annulus bundle around the loop corresponding to $w$ is orientation preserving but flip the boundary components due to $\pi_*u=2w$.
Therefore, the monodromy is isotopic to the map sending $z$ to $z^{-1}$.
This annulus bundle is isomorphic as an annulus bundle to (See the paragraph before Lemma 3.6 of \cite{Us09})
$$\mathbb{S}^1 \times \frac{\mathbb{R} \times \mathbb{A}}{(x+1,z) \sim (x,z^{-1})}$$
if $X$ is the smoothly trivial sphere bundle, and isomorphic to
$$\frac{\mathbb{R} \times \mathbb{S}^1 \times \mathbb{A}}{(x+1,e^{i\theta},z) \sim (x,e^{i \theta},e^{i \theta} z^{-1})}$$
if $X$ is the smoothly non-trivial sphere bundle.

Let $\overline{D}$ be another connected symplectic torus representing $c_1(X)$.
For $\overline{D}$, we can also define $\overline{J}$, $\overline{\pi}$, $\overline{\mathbb{T}^2}$, $\overline{t},\overline{u},\overline{v},\overline{w}$ as above.
Let $\tau:  \mathbb{T}^2 \to \overline{\mathbb{T}^2}$ be a diffeomorphism sending $v$ and $w$ to $\overline{v}$ and $\overline{w}$, respectively.
By construction, the pull-back annulus bundle $\tau^* (\overline{X}-P(\overline{D})) \to \mathbb{T}^2$ has the same monodromy (up to isotopy) as $X-P(D) \to \mathbb{T}^2$ over the one-skeleton.
The existence of an annulus bundle isomorphism from $X-P(D)$ to $\tau^* (X-\overline{D})$ covering the identity of $\mathbb{T}^2$ reduces to whether $X-P(D)$ and $\tau^* (\overline{X}-\overline{D})$ are isomorphic annulus bundle (covering some diffeomorphism of the base), which is true because there is only one class of isomorphic annulus bundle for a choice of monodromies over one skeleton (and it is explicitly described above in our case).
Therefore, we have a bundle isomorphism $F: X-P(D) \to X-P(\overline{D})$ covering $\tau$.
On the other hand, since the image of $\tau_* \circ \pi_*|_{H_1(D,\mathbb{Z})}$ equals the image of $\overline{\pi}_*|_{H_1(\overline{D},\mathbb{Z})}$,  there are two lifts of $\tau$ to $\tilde{\tau}_i: D \to \overline{D}$ such that $\overline{\pi} \circ \tilde{\tau}_i = \tau \circ \pi$, for $i=1,2$.
Then, there is a unique way, up to isotopy, to get a sphere bundle isomorphism $\tilde{F}: X \to X$ extending $F$ and $\tilde{\tau}_1$ (or, $F$ and $\tilde{\tau}_2$) by following the pseudo-holomorphic foliation.
In particular, we have $\tilde{F}(D)= \overline{D}$.

Using $\tilde{F}$, we can identify $\overline{D} \subset (X,\omega)$ with $D \subset (X,\tilde{F}^*\omega)$.
Proposition \ref{prop: torus type} will follow if we can find a symplectic deformation equivalence from $(X,D,\omega)$ to $(X,D,\tilde{F}^*\omega)$, which can be obtained by the following lemma.

\begin{lemma}\label{lem: Thurston trick}
Let  $\pi: (X,\omega_i,J_i) \to B$ be a symplectic surface bundle over surface such that
$J_i$ is $\omega_i$-compatible and fibers are $J_i$ holomorphic for both $i=0,1$.
Moreover, we assume the orientation of fibers induced by $J_0$ and $J_1$ are the same and the orientation of the total space induced by $\omega_0^2$ and $\omega_1^2$ are the same.
Assume $D \subset (X,\omega_i)$ is a $J_i$ holomorphic surface for $i=0,1$.
and $\pi|_D$ is submersive.
Then there is a smooth family of (possibly non-homologous) symplectic forms $\omega_t$ on $X$ making $D$ symplectic for all $t \in [0,1]$ joining $\omega_0$ and $\omega_1$.
\end{lemma}

\begin{proof}
Fix a point $p \in X$ and consider a non-zero tangent vector $v \in T_pX$  which does not lie in the vertical tangent sub-bundle $T_pX^{vert}$.
Since fibers are $J_i$ holomorphic, we have $Span\{v,J_iv\} \cap T_pX^{vert} =\{0\}$.
Choose a volume form (symplectic form) $\omega_B$ on $B$.
Since $\pi$ is a submersion, $\pi_* Span\{v,J_iv\} =T_{\pi(p)}B$.
Therefore, we have $\omega_B(\pi_*(v),\pi_*(J_iv)) \neq 0$.
By possibly changing the sign of $\omega_B$, we can assume $\omega_B(\pi_*(v),\pi_*(J_iv)) > 0$.
Moreover, this inequality is true for any $v \in T_pX$ not lying in $T_pX^{vert}$.
By continuity, $\omega_B(\pi_*(v),\pi_*(J_iv)) > 0$ for any $p \in X$ and any $v \in T_pX-T_pX^{vert}$ for both $i=0,1$.

Now, we apply the Thurston trick.
For any $K \ge 0$, we let $\omega_i^K=\omega_i + K \pi^* \omega_B$, which is clearly closed.
It is also non-degenerate because it is non-degenerate for the vertical tangent sub-bundle and for any $p \in X$, and any $v \in T_pX -T_pX^{vert}$,  we have
$\omega_i^K(v,J_i v)=\omega_i(v,J_iv)+K \omega_B(\pi_*(v),\pi_*(J_iv))>0$.
The first term being greater than zero is by compatibility and the second term being non-negative is due to $K \ge 0$ and the first paragraph.
Notice that $D$ is symplectic with respect to $\omega_i^K$ for both $i=0,1$ because $\pi|_D$ is submersive and $D$ is $J_i$-holomorphic.

Now, we consider $\omega^K_t=(1-t)\omega^K_0+t\omega^K_1$, which is clearly closed and non-degenerate for $TX^{vert}$.
For $v \in T_pX-T_pX^{vert}$, we have
$\omega^K_t(v,J_0v)=(1-t)\omega_0(v,J_0v)+t\omega_1(v,J_0v)+K\omega_B(\pi_*v,\pi_*J_0v)$.
We know that the first and the third terms on the right hand side are non-negative but we have no control on the second term.
However, there is a sufficiently large $K$ such that $\omega^K_t(v,J_0v) > 0$ for all $p \in X$ and $v \in T_pX -T_pX^{vert}$ and for all $t$ because
the sphere subbundle of $TX$ is compact.
By smoothening out the piecewise smooth family from $\omega_0$ to $\omega_0^K$, $\omega^K_t$ and from $\omega_1^K$ to $\omega_1$, we finish the proof.

\end{proof}

We remark that Lemma \ref{lem: Thurston trick} can be viewed as a relative version of Proposition 4.4 in \cite{Mc98} in dimension four.

\subsection{Proof of Theorem \ref{thm: symplectic deformation class=homology classes}}

We are ready to prove Theorem \ref{thm: symplectic deformation class=homology classes}.

\begin{proof}[Proof of Theorem \ref{thm: symplectic deformation class=homology classes}]



 Let $(X^i,D^i,\omega^i)$ be symplectic log Calabi-Yau surfaces for $i=1,2$, which are homological equivalent via a diffeomorphism $\Phi$.


Let $\{e_1, \dots, e_{\beta}\}$ be a maximal set of pairwisely orthogonal non-toric exceptional classes in $X$.
We can choose an almost complex structure $J^1$ (possibly after deforming $D^1$) such that $D^1$ is $J^1$-holomorphic and all $e_j$ has
embedded $J^1$-holomorphic representative, by Lemma \ref{lem: existence of nice J-sphere}.
Since $(X^1,D^1,\omega^1)$ and $(X^2,D^2,\omega^2)$ are homological equivalent via $\Phi$, $\{ \Phi_*(e_j)\}$ is a maximal set of pairwisely orthogonal non-toric exceptional classes.
We can find an $\omega^2$-tamed almost
complex structure (possibly after deforming $D^2$) $J^2$
such that $D^2$ is $J^2$-holomorphic and the $\Phi_*(e_j)$ has embedded $J^2$-holomorphic representative.
After blowing down  the $J^i$-holomorphic representatives of $e_j$, and $\Phi_*(e_j)$ for all $1 \le j \le \beta$,
we obtain two symplectic log CY surfaces $(\overline{X^1},\overline{D^1},\overline{\omega^1})$
and $(\overline{X^2},\overline{D^2},\overline{\omega^2})$.

Clearly, $(\overline{X^1},\overline{D^1},\overline{\omega^1})$
and $(\overline{X^2},\overline{D^2},\overline{\omega^2})$ are homological equivalent for some natural choice of diffeomorphism $\overline{\Phi}$.
Now, a component in $\overline{D^1}$ is exceptional if and only if the corresponding component in $\overline{D^2}$ is exceptional.
By Lemma \ref{lemma: fundamental theorem + minimal model},  we pass
to  minimal models $(\overline{X^i_b},\overline{D^i_b},\overline{\omega^i_b})$ by toric blow-downs. 
By identifying $\overline{X^1_b}$ and $\overline{X^2_b}$ with a natural choice of diffeomorphism $\overline{\Phi_b}$, the homology classes of the components of $\overline{D^1_b}$ and $\overline{D^2_b}$ are the same.

By Proposition 1.2.15 of \cite{McOp13} or Theorem 2.9 of \cite{DoLiWu14},  up to a D-symplectic homotopy (ie. a deformation of $\overline{\omega^2_b}$ keeping $\overline{D^2_b}$ symplectic),
we can assume $[\overline{\omega^1_b}]=\overline{\Phi_b}^*[\overline{\omega^2_b}]$.
Therefore, $\overline{X^1_b}$ and $\overline{X^2_b}$ are actually symplectomorphic (\cite{Ta95}, \cite{LaMc96-2})
and we thus can choose $\overline{\Phi_b}$ to be a symplectomorphism from $(\overline{X^1_b},\overline{\Phi_b}^{-1}(\overline{D^2_b}),\overline{\omega^1_b})$ to $(\overline{X^2_b},\overline{D^2_b},\overline{\omega^2_b})$.
Therefore, we conclude that
$(\overline{X^1_b},\overline{D^1_b},\overline{\omega^1_b})$ and $(\overline{X^2_b},\overline{D^2_b},\overline{\omega^2_b})$
 are symplectic deformation equivalent, by applying Proposition \ref{prop: isotopy class in minimal model} to $(\overline{X^1_b},\overline{D^1_b},\overline{\omega^1_b})$ and $(\overline{X^1_b},\overline{\Phi_b}^{-1}(\overline{D^2_b}),\overline{\omega^1_b})$.
 Further, by Lemma \ref{lem: as a marked divisor},  they are  $D-$symplectic deformation equivalent.

Now we record the sequence of  non-toric and toric blow-downs by markings   $\overline{D^1_b}$ and $\overline{D^2_b}$.
 As marked divisors, they are
$D-$symplectic deformation equivalent by Lemma \ref{lem: marking adding/shrinking}.
Finally, by Proposition \ref{prop: marked blowdown} (and viewing unmarked divisors as marked divisors without markings), $(X^1,D^1,\omega^1)$ is $D-$symplectic deformation equivalent to $(X^2,D^2,\omega^2)$, and hence
symplectic deformation equivalent to $(X^2,D^2,\omega^2)$ by Lemma \ref{lem: as a marked divisor}.
Tracing the steps, we see that the symplectomorphism in the symplectic deformation equivalence between $(X^1,D^1,\omega^1)$ and $(X^2,D^2,\omega^2)$
has the same homological effect as $\Phi$.



Now, assume $(X^1,D^1,\omega^1)$ is {\it strictly} homological equivalent to $(X^2,D^2,\omega^2)$ via a diffeomorphism $\Phi$.
It means that $\Phi$ is a homological equivalence and $\Phi^*[\omega^2]=[\omega^1]$.
We first note that, up to symplectic isotopy of $D^1$ and $D^2$, which preserves the strict D-symplectic deformation class (Lemma \ref{lem: perturbation}),
we can assume $D^i$ are $\omega^i$-orthogonal.
We have shown that there is a $D-$symplectic homotopy $(X^1,D^1,\omega^1_t)$ of $(X^1,D^1,\omega^1)$
and a symplectomophism $\Psi: (X^1,D^1,\omega^1_1) \to (X^2,D^2,\omega^2)$ with the same homological effect as $\Phi$.
Therefore, we have $[\omega^1]=\Phi^*[\omega^2]=\Psi^*[\omega^2]=[\omega^1_1]$.
By Theorem 1.2.12 of \cite{McOp13}, $\omega^1_t$ can be chosen such that $[\omega^1_t]$ is constant for all $t$.
By Corollary 1.2.13 of \cite{McOp13}, there is a symplectic isotopy $(X^1,D^1_t,\omega^1)$ such that $D^1_0=D^1$ and $(X^1,D^1_1,\omega^1)$ is symplectomorphic to
$(X^1,D^1,\omega^1_1)$ and hence to $(X^2,D^2,\omega^2)$.
Therefore, the result follows.

\end{proof}

In the case $X^1=X^2=X$, Theorem \ref{thm: symplectic deformation class=homology classes} implies the symplectic deformation class of $(X,D,\omega)$ is uniquely determined by the homology classes $\{[C_j]\}_{j=1}^k$ modulo the action of diffeomorphism on $H_2(X,\mathbb{Z})$.
The fact the the homology classes of $D$ completely determine the symplectic deformation equivalent class can be regarded as in the same spirit of Torelli type theorems in a weak sense.

If $(X^1,\omega^1)=(X^2,\omega^2)=(X, \omega)$, we can take the strict homological equivalence to be identity and hence
the symplectomorphism from $(X,D^1,\omega)$ to the time-one end of the symplectic isotopy of $(X,D^2,\omega)$
in Theorem \ref{thm: symplectic deformation class=homology classes} has trivial homological action.
Therefore, the number of symplectic isotopy classes of homological equivalent log Calabi-Yau surfaces in $(X,\omega)$ is bounded above by the
number of connected components of Torelli part of the symplectomorphism group of $(X,\omega)$.



\begin{thebibliography}{99}



\bibitem{AbMc00}
M.~Abreu and D.~McDuff
\newblock Topology of symplectomorphism groups of rational ruled surfaces.
\newblock {\em J. Amer. Math. Soc.} 13: 971-1009, 2000.


\bibitem{Au07}
D. Auroux.
\newblock Mirror symmetry and T-duality in the complement of an anticanonical divisor.
\newblock {\em J. Gokova Geom. Topol. GGT.} 1: 51-91, 2007.

\bibitem{DoLiWu14}
J.~Dorfmeister, T-J. Li and W. Wu.
\newblock Stability and existence of surfaces in symplectic 4-manifolds with $b^+=1$.
\newblock arXiv:1407.1089, 2014.



\bibitem{Fr15}
R. Friedman.
\newblock On the geometry of anticanonical pairs.
\newblock arXiv:1502.02560, 2015.


\bibitem{Gom95}
R.E.~Gompf.
\newblock A new construction of symplectic manifolds.
\newblock {\em Ann. of Math.}, 143(3):527-595, 1995.

\bibitem{GrHaKe11}
M.~Gross, P.~Hacking and S.~Keel.
\newblock Mirror symmetry for log Calabi-Yau surfaces I.
\newblock arXiv:1106.4977, 2011.

\bibitem{GrHaKe12}
M.~Gross, P.~Hacking and S.~Keel.
\newblock Moduli of surfaces with an anti-canonical cycle.
\newblock arXiv:1211.6367, 2012.

\bibitem{HoLiSi97}
H.~Hofer, V.~Lizan and J.-C.~Sikorav.
\newblock On genericity for holomorphic curves in four-dimensional almost-complex manifolds.
\newblock {\em J. Geom. Anal.}, 7(1):149-159, 1997.

\bibitem{IvSh99}
S. Ivashkovich and V. Shevchishin.
\newblock Structure of the moduli space in a neighborhood of a cusp-curve and meromorphic hulls.
\newblock {\em Invent. Math.} 136(3):571-602, 1999.


\bibitem{LaMc96-2}
F.~Lalonde and D.~McDuff.
\newblock J-curves and the classification of ruled symplectic 4-manifolds.
\newblock {\em Contact and symplectic geometry Publ. Newton Inst. 8}, (Cambridge, 1994), 3-42, Cambridge Univ. Press, Cambridge, 1996




\bibitem{LiMa14}
T-J. Li and C.Y. Mak.
\newblock Symplectic Divisorial Capping in Dimension 4.
\newblock arXiv:1407.0564, 2014.

\bibitem{LiMaYa14}
T-J. Li, C.Y. Mak and K. Yasui.
\newblock Uniruled caps and Calabi-Yau caps.
\newblock arXiv:1412.3208, 2014.



\bibitem{LiWu12}
T-J.~Li and W. Wu.
\newblock Lagrangian spheres, symplectic surfaces and symplectic mapping class groups.
\newblock {\em Geom. and Topo.}, 16:1121-1196, 2012.

\bibitem{LiZh11}
T-J. Li and W. Zhang.
\newblock Additive and relative Kodaira dimension.
\newblock {\em Geometry and Analysis, No.2} 18:103-135, Adv. Lect. Math.(ALM), Int. Press., 2011.



\bibitem{Liu96}
A. Liu.
\newblock Some new applications of general wall crossing formula, Gompf's conjecture and its applications.
\newblock {\em Math. Res. Lett.} 3(5): 569-585, 1996.

\bibitem{Lo81}
E.~Looijenga
\newblock Rational surfaces with an anti-canonical cycle.
\newblock {\em Ann. of Math.} 114(2): 267-322, 1981.

\bibitem{Mc98}
D.~McDuff
\newblock From symplectic deformation to isotopy.
\newblock {\em First Int. Press Lect. Ser.} I: 85-99, Topics in symplectic 4-manifolds (Irvine, CA, 1996),
Int. Press, Cambridge, MA, 1998.

\bibitem{McSa98}
D.~McDuff and D.~Salamon
\newblock Introduction to symplectic topology.
\newblock {\em Oxford Mathematics Monographs} 2nd edition, the Clarendon Press,
Oxford University Press, New York ,1998.

\bibitem{McSa12}
D.~McDuff and D.~Salamon
\newblock $J$ holomorphic curves and symplectic topology.
\newblock {\em American Mathematics Society Colloquium Publication 52} second edition, American Mathematical Society, Providence, RI, 2012.

\bibitem{McOp13}
D.~McDuff and E.~Opshtein
\newblock Nongeneric J-holomorphic curves and singular inflation.
\newblock {\em Algebr. Geom. Topol.} 15: 231-286, 2015.

\bibitem{McL12}
M.~McLean.
\newblock The growth rate of symplectic homology and affine varieties.
\newblock {\em Geometric And Functional Analysis.}, 22(2): 369-442, 2012.


\bibitem{OhOn96}
H.~Ohta and K.~Ono.
\newblock   Notes on symplectic 4-manifolds with $b^+=1$,  II.
\newblock {\em Int. J. Math. } 7: 755-770, 1996.

\bibitem{OhOn05}
H.~Ohta and K.~Ono.
\newblock Simple singularities and symplectic fillings.
\newblock {\em J. Diff. Geom.}, 69(1):1-42, 2005.


\bibitem{Pa13}
J.~Pascaleff
\newblock On the symplectic cohomology of log Calabi-Yau surfaces.
\newblock arXiv:1304.5298, 2013.

\bibitem{Pi08}
M.~Pinsonnault.
\newblock Maximal compact tori in the Hamiltonian group of 4-dimensional symplectic manifolds.
\newblock {\em J. Mod. Dyn.}, 2(3): 431-455, 2008.

\bibitem{Sa13}
D.~Salamon.
\newblock Uniqueness of symplectic structures.
\newblock {\em Acta Math. Vietnam.}, 38(1): 123-144, 2013.


\bibitem{SiTi05}
B.~Siebert and G.~Tian.
\newblock On the holomorphicity of genus two Lefschetz fibrations.
\newblock {\em Ann. of Math.}, 2(161): 959-1020, 2005.

\bibitem{Si03}
J.C.~Sikorav.
\newblock The gluing construction for normally generic J-holomorphic curves.
\newblock {\em Fields Inst. Commum.} 35: 175-199, Symplectic and contact topology: interactions and perspectives (Toronto, ON/Montreal, QC, 2001), Amer. Math.
Soc., Providence, RI, 2003.


\bibitem{Ta95}
C.H.~Taubes.
\newblock The Seiberg-Witten and Gromov invariants.
\newblock {\em Math. Res. Lett.} 2(2): 221-238, 1995.

\bibitem{Us09}
M.~Usher
\newblock Kodaira dimension and symplectic sums.
\newblock {\em Comment. Math. Helv.} 1(84): 57-85, 2009.

\end{thebibliography}
\end{document}